\documentclass{amsart}
\usepackage{amsmath,amsthm,amssymb,latexsym,epic,bbm,comment,mathrsfs}
\usepackage{graphicx,enumerate,stmaryrd,color,tikz,mathdots}
\usepackage[all,2cell]{xy}
\xyoption{2cell}
\usetikzlibrary{arrows}
\usepackage[T1]{fontenc}

\newtheorem{theorem}{Theorem}

\newtheorem{proposition}[theorem]{Proposition}
\newtheorem{conjecture}[theorem]{Conjecture}

\newtheorem{remark}[theorem]{Remark}
\usepackage[all]{xy}
\usepackage[active]{srcltx}
\usepackage[parfill]{parskip}
\usepackage{enumerate}

\newcommand{\tto}{\twoheadrightarrow}

\begin{document}
\title[Combinatorics of monoidal actions]
{Combinatorics of monoidal actions\\ in Lie-algebraic context}

\author[V.~Mazorchuk]{Volodymyr Mazorchuk}
\address{Department of Mathematics\\ 
Uppsala University\\ 
Box. 480\\
SE-75106\\ 
Uppsala\\ 
SWEDEN}
\email{mazor\symbol{64}math.uu.se}

\author[X.~Zhu]{Xiaoyu Zhu}
\address{School of Mathematics and Statistics\\
Ningbo University\\ 
Ningbo, Zhejiang\\ 
315211\\
P.R. CHINA}
\email{zhuxiaoyu1\symbol{64}nbu.edu.cn}

\begin{abstract}
This paper is, essentially, a survey related to the problem of
understanding the combinatorics of the action of the 
monoidal category of finite dimensional modules over a
simple finite dimensional Lie algebra on various
categories of Lie algebra modules. A special attention is payed
to the Lie algebras $\mathfrak{sl}_2$ and $\mathfrak{sl}_3$.
A few new general results are collected at the end. 
\end{abstract}

\maketitle

\section{Introduction}\label{s1}

Monoidal categories and their actions are important mathematical 
concepts that found a wide range of applications to various 
areas of modern mathematics and theoretical physics, especially
in representation theory, quantum mechanics, topological quantum 
field theory and several further areas that actively use the 
general idea of categorification, see, for example \cite{Co,CP,EGNO}
and references therein.

Lie groups, Lie algebras and their module categories serve as
a rich source of various types of monoidal categories. The most
prominent of these is the category of finite dimensional 
modules over a Lie group or a Lie algebra. As in 
this paper we will be mostly focusing on the setup of Lie algebras,
let $\mathfrak{g}$ be a Lie algebra over some field $\Bbbk$.
Then the universal enveloping algebra $U(\mathfrak{g})$
is, naturally, a Hopf algebra, which endows the category
$\mathscr{C}:=\mathfrak{g}\text{-}\mathrm{fdmod}$ of all finite dimensional
$\mathfrak{g}$-modules with the natural structure of a
(symmetric) monoidal category. 

The monoidal category $\mathscr{C}$ acts, in the obvious way, 
on the category  $\mathfrak{g}\text{-}\mathrm{Mod}$ of all
$\mathfrak{g}$-modules and this action restricts to many 
natural subcategories of $\mathfrak{g}\text{-}\mathrm{Mod}$.
From the point of view of representation theory, it is natural 
to pose the problem of classification, up to equivalence, of
such actions. This problem is the main motivation for most
of the results presented in this paper. Taking literally, the
problem seems too hard. However, it contains many interesting 
subproblems and special cases.

The case of the Lie algebra $\mathfrak{sl}_2$ is a natural 
starting point. Here it turns out that the combinatorics of
``simple'' actions of the category $\mathfrak{sl}_2\text{-}\mathrm{fdmod}$
on various categories of Lie algebra modules 
is described by infinite Dynkin diagrams. This combinatorics is
only a rough invariant, it does not identify an action of 
$\mathfrak{sl}_2\text{-}\mathrm{fdmod}$, up to equivalence.
In Section~\ref{s3} of the present paper we survey the results of \cite{MZ1}
where the action of $\mathfrak{sl}_2\text{-}\mathrm{fdmod}$
in various Lie-algebraic contexts was studied in detail.

The next natural step is the Lie algebra $\mathfrak{sl}_3$,
where the situation is much more complicated. Nevertheless, the
combinatorics of  ``simple'' actions of the category 
$\mathfrak{sl}_3\text{-}\mathrm{fdmod}$ on those categories of 
$\mathfrak{sl}_3$-modules which are ``generated'' by a simple 
(but not necessarily
finite dimensional) $\mathfrak{sl}_3$-module can be classified in
terms of eight ``two-dimensional'' graphs. This is done in 
\cite{MZ2} and surveyed in Section~\ref{s4}.

In Section~\ref{s5} we outline the setup for similar questions
in the general case, collect a few starting general results and
describe our expectations.

Section~\ref{s2} serves as a motivating introduction. It 
recaps the classical results on Dynkin diagrams and related
classifications, including the $ADE$-type classifications that
use simply laced Dynkin diagrams.

\vspace{5mm}

{\bf Acknowledgements.}

The first author is partially supported by the Swedish Research Council, 
Grant No. 2021-03731. The second author is partially supported by the 
Zhejiang Provincial Natural Science Foundation of China, Grant No. LQN25A010023.
Most of the results in this paper were presented at the 
XLII Workshop on Geometric Methods in Physics, which took place in Bia{\l}ystok
in July 2025. We thank the organizers of the Workshop for this
opportunity.
\vspace{5mm}

\section{Dynkin diagrams and various classifications}\label{s2}

\subsection{Dynkin diagrams}\label{s2.1}

Recall the following list of {\em Dynkin diagrams}:

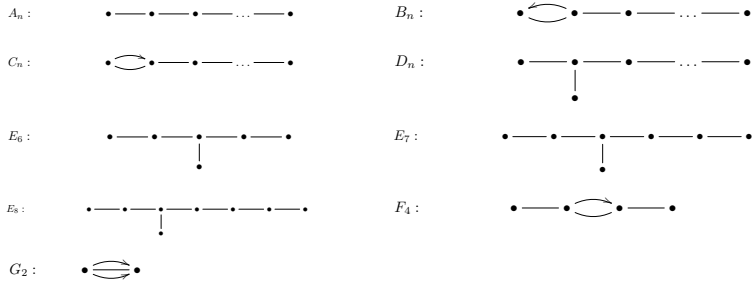
\resizebox{5cm}{!}{
$
\xymatrix{
A_n:&&\bullet\ar@{-}[r]&\bullet\ar@{-}[r]&\bullet\ar@{-}[r]&\dots\ar@{-}[r]&\bullet
}\qquad\qquad\qquad
$
}
\resizebox{5cm}{!}{
$
\xymatrix{
B_n:&&\bullet\ar@/_2mm/@{-}[r]&\bullet\ar@{-}[r]\ar@/_2mm/[l]&\bullet\ar@{-}[r]&\dots\ar@{-}[r]&\bullet
}
$
}

\resizebox{5cm}{!}{
$
\xymatrix{
C_n:&&\bullet\ar@/_2mm/@{-}[r]\ar@/^2mm/[r]&\bullet\ar@{-}[r]&\bullet\ar@{-}[r]&\dots\ar@{-}[r]&\bullet
}\qquad\qquad\qquad
$
}
\resizebox{5cm}{!}{
$
\xymatrix@R=4mm{
D_n:&&\bullet\ar@{-}[r]&\bullet\ar@{-}[r]\ar@{-}[d]&\bullet\ar@{-}[r]&\dots\ar@{-}[r]&\bullet\\
&&&\bullet&&&
}
$
}

\resizebox{5cm}{!}{
$
\xymatrix@R=4mm{
E_6:&&\bullet\ar@{-}[r]&\bullet\ar@{-}[r]&\bullet\ar@{-}[r]\ar@{-}[d]&
\bullet\ar@{-}[r]&\bullet\\
&&&&\bullet&&
}\qquad\qquad\qquad
$
}
\resizebox{5cm}{!}{
$
\xymatrix@R=4mm{
E_7:&&\bullet\ar@{-}[r]&\bullet\ar@{-}[r]&\bullet\ar@{-}[r]\ar@{-}[d]&
\bullet\ar@{-}[r]&\bullet\ar@{-}[r]&\bullet\\
&&&&\bullet&&
}
$
}

\resizebox{5cm}{!}{
$
\xymatrix@R=4mm{
E_8:&&\bullet\ar@{-}[r]&\bullet\ar@{-}[r]&\bullet\ar@{-}[r]\ar@{-}[d]&
\bullet\ar@{-}[r]&\bullet\ar@{-}[r]&\bullet\ar@{-}[r]&\bullet\\
&&&&\bullet&&
}\qquad\qquad\qquad
$
}
\resizebox{4cm}{!}{
$
\xymatrix@R=4mm{
F_4:&&\bullet\ar@{-}[r]&\bullet\ar@{-}@/_2mm/[r]\ar@/^2mm/[r]&\bullet\ar@{-}[r]&\bullet
}
$
}

\resizebox{2cm}{!}{
$
\xymatrix@R=4mm{
G_2:&\bullet\ar@{-}[r]\ar@/_2mm/[r]\ar@/^2mm/[r]&\bullet
}
$
}

The number of vertices of a diagram is called the {\rm rank}
and is used as the subscript in the notation.
The diagrams $A_n$, $D_n$ and $E_6$, $E_7$ and $E_8$, that is,
those Dynkin diagrams that do not have any multiple edges,
are called {\em simply laced}.

Recall also the following list of {\em affine Dynkin diagrams}:

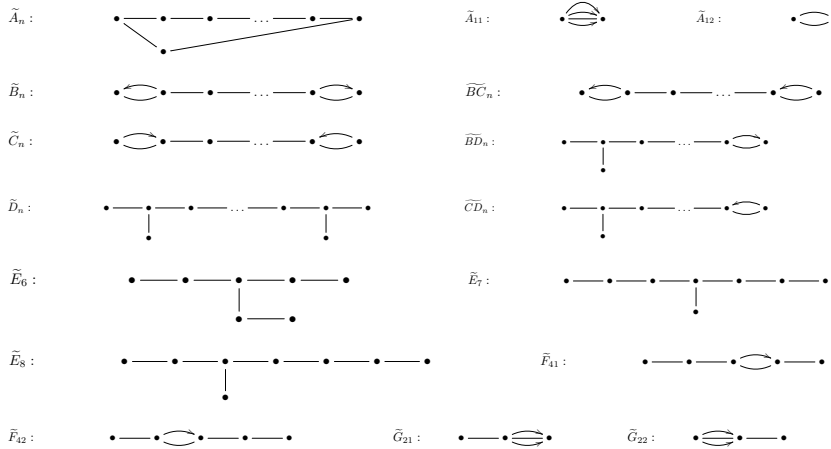
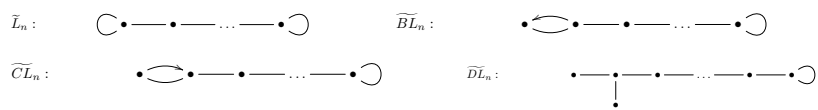
\resizebox{6cm}{!}{
$
\xymatrix@R=4mm{
\widetilde{A}_n:&&\bullet\ar@{-}[r]&\bullet\ar@{-}[r]&
\bullet\ar@{-}[r]&\dots\ar@{-}[r]&\bullet\ar@{-}[r]&\bullet\\
&&&\bullet\ar@{-}[lu]\ar@{-}[rrrru]&&&
}\qquad\qquad\qquad
$
}
\resizebox{3cm}{!}{
$
\xymatrix@R=4mm{
\widetilde{A}_{11}:&&
\bullet\ar@{-}[r]\ar@/_2mm/[r]\ar@/^2mm/[r]\ar@/^5mm/[r]&\bullet
}\qquad\qquad\qquad
$
}
\resizebox{3cm}{!}{
$
\xymatrix@R=4mm{
\widetilde{A}_{12}:&&\bullet\ar@/_2mm/@{-}[r]\ar@/^2mm/@{-}[r]&\bullet
}\qquad\qquad\qquad
$
}

\resizebox{6cm}{!}{
$
\xymatrix@R=4mm{
\widetilde{B}_n:&&\bullet\ar@/_2mm/@{-}[r]&
\bullet\ar@{-}[r]\ar@/_2mm/[l]&\bullet\ar@{-}[r]&\dots\ar@{-}[r]&
\bullet\ar@/_2mm/@{-}[r]\ar@/^2mm/[r]&\bullet
}\qquad\qquad\qquad
$
}
\resizebox{6cm}{!}{
$
\xymatrix@R=4mm{
\widetilde{BC}_n:&&\bullet\ar@/_2mm/@{-}[r]&
\bullet\ar@{-}[r]\ar@/_2mm/[l]&\bullet\ar@{-}[r]&
\dots\ar@{-}[r]&\bullet\ar@/_2mm/@{-}[r]&\bullet\ar@/_2mm/[l]
}\qquad\qquad\qquad
$
}

\resizebox{6cm}{!}{
$
\xymatrix@R=4mm{
\widetilde{C}_n:&&\bullet\ar@/_2mm/@{-}[r]\ar@/^2mm/[r]&
\bullet\ar@{-}[r]&\bullet\ar@{-}[r]&\dots\ar@{-}[r]&
\bullet\ar@/_2mm/@{-}[r]&\bullet\ar@/_2mm/[l]
}\qquad\qquad\qquad
$
}
\resizebox{6cm}{!}{
$
\xymatrix@R=4mm{
\widetilde{BD}_n:&&\bullet\ar@{-}[r]&
\bullet\ar@{-}[r]\ar@{-}[d]&\bullet\ar@{-}[r]&
\dots\ar@{-}[r]&\bullet\ar@/_2mm/@{-}[r]\ar@/^2mm/[r]&\bullet\\
&&&\bullet&&&&&&
}\qquad\qquad\qquad
$
}

\resizebox{6cm}{!}{
$
\xymatrix@R=4mm{
\widetilde{D}_n:&&\bullet\ar@{-}[r]&
\bullet\ar@{-}[r]\ar@{-}[d]&\bullet\ar@{-}[r]&\dots\ar@{-}[r]&
\bullet\ar@{-}[r]&\bullet\ar@{-}[r]\ar@{-}[d]&\bullet\\
&&&\bullet&&&&\bullet&
}\qquad\qquad\qquad
$
}
\resizebox{6cm}{!}{
$
\xymatrix@R=4mm{
\widetilde{CD}_n:&&\bullet\ar@{-}[r]&
\bullet\ar@{-}[r]\ar@{-}[d]&\bullet\ar@{-}[r]&
\dots\ar@{-}[r]&\bullet\ar@/_2mm/@{-}[r]&\bullet\ar@/_2mm/[l]\\
&&&\bullet&&&&&&
}\qquad\qquad\qquad
$
}

\resizebox{6cm}{!}{
$
\xymatrix@R=4mm{
\widetilde{E}_6:&&\bullet\ar@{-}[r]&\bullet\ar@{-}[r]&\bullet\ar@{-}[r]\ar@{-}[d]&
\bullet\ar@{-}[r]&\bullet\\
&&&&\bullet\ar@{-}[r]&\bullet&
}\qquad\qquad\qquad
$
}
\resizebox{5cm}{!}{
$
\xymatrix@R=4mm{
\widetilde{E}_7:&&\bullet\ar@{-}[r]&\bullet\ar@{-}[r]&
\bullet\ar@{-}[r]&\bullet\ar@{-}[r]\ar@{-}[d]&
\bullet\ar@{-}[r]&\bullet\ar@{-}[r]&\bullet\\
&&&&&\bullet&&
}
$
}

\resizebox{7cm}{!}{
$
\xymatrix@R=4mm{
\widetilde{E}_8:&&\bullet\ar@{-}[r]&\bullet\ar@{-}[r]&\bullet\ar@{-}[r]\ar@{-}[d]&
\bullet\ar@{-}[r]&\bullet\ar@{-}[r]&\bullet\ar@{-}[r]&\bullet\\
&&&&\bullet&&
}\qquad\qquad\qquad
$
}
\resizebox{5cm}{!}{
$
\xymatrix@R=4mm{
\widetilde{F}_{41}:&&\bullet\ar@{-}[r]&
\bullet\ar@{-}[r]&\bullet\ar@{-}@/_2mm/[r]\ar@/^2mm/[r]&\bullet\ar@{-}[r]&\bullet
}\qquad\qquad\qquad
$
}

\resizebox{5cm}{!}{
$
\xymatrix@R=4mm{
\widetilde{F}_{42}:&&\bullet\ar@{-}[r]&\bullet\ar@{-}@/_2mm/[r]\ar@/^2mm/[r]&\bullet\ar@{-}[r]&\bullet\ar@{-}[r]&\bullet
}\qquad\qquad\qquad
$
}
\resizebox{3cm}{!}{
$
\xymatrix@R=4mm{
\widetilde{G}_{21}:&\bullet\ar@{-}[r]&
\bullet\ar@{-}[r]\ar@/_2mm/[r]\ar@/^2mm/[r]&
\bullet
}\qquad\qquad
$
}
\resizebox{3cm}{!}{
$
\xymatrix@R=4mm{
\widetilde{G}_{22}:&\bullet\ar@{-}[r]\ar@/_2mm/[r]\ar@/^2mm/[r]&
\bullet\ar@{-}[r]&\bullet
}\qquad\qquad
$
}

\resizebox{5cm}{!}{
$
\xymatrix@R=4mm{
\widetilde{L}_{n}:&&\bullet\ar@{-}[r]\ar@{-}@(ul,dl)[]&
\bullet\ar@{-}[r]&\dots\ar@{-}[r]&\bullet\ar@{-}@(ur,dr)[]
}\qquad\qquad\qquad
$
}
\resizebox{6cm}{!}{
$
\xymatrix@R=4mm{
\widetilde{BL}_n:&&\bullet\ar@/_2mm/@{-}[r]&
\bullet\ar@{-}[r]\ar@/_2mm/[l]&\bullet\ar@{-}[r]&\dots\ar@{-}[r]&
\bullet\ar@{-}@(ur,dr)[]
}\qquad\qquad\qquad
$
}

\resizebox{6cm}{!}{
$
\xymatrix@R=4mm{
\widetilde{CL}_n:&&\bullet\ar@/_2mm/@{-}[r]\ar@/^2mm/[r]&
\bullet\ar@{-}[r]&\bullet\ar@{-}[r]&\dots\ar@{-}[r]&\bullet\ar@{-}@(ur,dr)[]
}\qquad\qquad\qquad
$
}
\resizebox{6cm}{!}{
$
\xymatrix@R=4mm{
\widetilde{DL}_n:&&\bullet\ar@{-}[r]&
\bullet\ar@{-}[r]\ar@{-}[d]&\bullet\ar@{-}[r]&\dots\ar@{-}[r]&
\bullet\ar@{-}[r]&\bullet\ar@{-}@(ur,dr)\\
&&&\bullet&&&&&
}\qquad\qquad\qquad
$
}

We note that in types $D$ and $E$ there is a natural bijection
between the corresponding Dynkin diagrams and affine Dynkin diagrams.
In type $A$ this bijection exists, with the exception for rank $1$.
This can be amended by the convention to treat $\widetilde{A}_{12}$
as a simply laced diagram.

\subsection{Classifications that use Dynkin diagrams}\label{s2.2}

There are a few famous classification results in mathematics that
use Dynkin diagrams. For example, Dynkin diagrams classify
the following mathematical objects, see \cite{Hu75,EW,Kn}:
\begin{itemize}
\item irreducible finite root systems;
\item simple complex finite dimensional Lie algebras;
\item simple algebraic groups over an algebraically closed
field, up to isogeny;
\item simply connected complex Lie groups which are simple modulo
centers;
\item simply connected compact Lie groups which are simple modulo
centers.
\end{itemize}

As a related classification one could also mention that of 
irreducible Weyl groups (the caveat here is that this
classification is not bijective as, for instance,
the Weyl groups for types $B_n$ and $C_n$ are isomorphic).

\subsection{$ADE$ classifications}\label{s2.3}

It is quite remarkable that a significantly larger variety of
mathematical objects admits a classification in terms of 
simply laced Dynkin diagrams, the so called {\em $ADE$-classification}.
For example, simply laced Dynkin diagrams classify
the following mathematical objects:
\begin{itemize}
\item simply laced root systems;
\item simply laced Lie algebras;
\item finite subgroups of $\mathrm{SL}_2(\mathbb{C})$, see \cite{Mc,Re};
\item discrete subgroups of $\mathrm{SU}(2)$
(the so-called {\em McKay correspondence}), see \cite{Mc,Re};
\item underlying unoriented graphs for quivers
of finite representation type, see \cite{Ga};
\item simple hypersurface singularities (a.k.a. DuVal or Kleinian
singularities), see \cite{Ar,St};
\item connected finite graphs for which the spectral radius of the adjacency
matrix is less than $2$, see \cite{Sm};
\item connected finite graphs for which the spectral radius of the adjacency
matrix equals $2$, see \cite{Sm};
\item minimal and $A^{(1)}_1$-conformal 
invariant theories, see \cite{CIZ};
\item simple transitive 2-representations of Soergel bimodules with 
non-ext\-reme apex in finite dihedral types, see \cite{KMMZ,MT}.
\end{itemize}
Additionally, one should also mention the $ADE$-type classification of
$\mathfrak{sl}_2$ conformal field theories in \cite{KO}.

\subsection{Infinite Dynkin diagrams}\label{s2.4}

The papers \cite{HPR,HPR2} propose the following 
infinite generalizations of Dynkin diagrams
and discuss their relevance in the context of representation theory
of associative algebras:

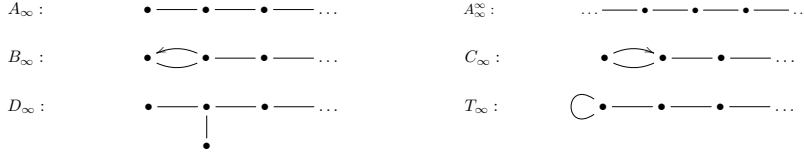
\resizebox{6cm}{!}{
$
\xymatrix@R=4mm{
A_\infty:&&\bullet\ar@{-}[r]&\bullet\ar@{-}[r]&
\bullet\ar@{-}[r]&\dots
}\qquad\qquad\qquad
$
}
\resizebox{6cm}{!}{
$
\xymatrix@R=4mm{
A_\infty^\infty:&&\dots\ar@{-}[r]&\bullet\ar@{-}[r]&\bullet\ar@{-}[r]&
\bullet\ar@{-}[r]&\dots
}\qquad\qquad\qquad
$
}

\resizebox{6cm}{!}{
$
\xymatrix@R=4mm{
B_\infty:&&\bullet\ar@/_2mm/@{-}[r]&\bullet\ar@{-}[r]\ar@/_2mm/[l]&
\bullet\ar@{-}[r]&\dots
}\qquad\qquad\qquad
$
}
\resizebox{6cm}{!}{
$
\xymatrix@R=4mm{
C_\infty:&&\bullet\ar@/_2mm/@{-}[r]\ar@/^2mm/[r]&\bullet\ar@{-}[r]&
\bullet\ar@{-}[r]&\dots
}\qquad\qquad\qquad
$
}

\resizebox{6cm}{!}{
$
\xymatrix@R=4mm{
D_\infty:&&\bullet\ar@{-}[r]&\bullet\ar@{-}[r]&
\bullet\ar@{-}[r]&\dots\\
&&&\bullet\ar@{-}[u]&
}\qquad\qquad\qquad
$
}
\resizebox{6cm}{!}{
$
\xymatrix@R=4mm{
T_\infty:&&\bullet\ar@{-}@(ul,dl)[]\ar@{-}[r]&\bullet\ar@{-}[r]&
\bullet\ar@{-}[r]&\dots
}\qquad\qquad\qquad
$
}

\section{$\mathfrak{sl}_2$-combinatorics}\label{s3}

\subsection{McKay correspondence as inspiration}\label{s3.1}

In the case of finite subgroups of $\mathrm{SL}_2(\mathbb{C})$,
the McKay correspondence mentioned in Subsection~\ref{s2.3}
works as follows: let $\mathscr{C}$ be the monoidal category of
all finite dimensional $\mathrm{SL}_2(\mathbb{C})$-module. 
The monoidal structure here is the obvious one in which the 
trivial module serves as the monoidal unit and tensor product 
is given by tensoring over $\mathbb{C}$ and using the diagonal 
action of the group on this tensor product. The category 
$\mathscr{C}$ is generated, as a monoidal category, by the 
natural ($2$-dimensional) module $V=\mathbb{C}^2$.

Let $G$ be a finite subgroup of $\mathrm{SL}_2(\mathbb{C})$. Then
$\mathscr{C}$ acts on the category of finite dimensional 
$G$-modules in the obvious way, that is, using the tensor 
product and restriction. Let $L_1,L_2,\dots, L_k$ be a complete 
and irredundant list of representatives of the isomorphism classes
of simple $G$-modules. As any finite dimensional $G$-module
is completely reducible, the module $V\otimes_{\mathbb{C}}L_i$
is determined uniquely, up to isomorphism, by the composition
multiplicities $[V\otimes_{\mathbb{C}}L_i:L_j]$. Since $V$
is self-dual, these multiplicities are symmetric in the sense that
\begin{displaymath}
[V\otimes_{\mathbb{C}}L_i:L_j]=[V\otimes_{\mathbb{C}}L_j:L_i], 
\end{displaymath}
for all $i,j$, and hence can be represented as an unoriented graph
with vertices the $L_i$'s. The point of the McKay correspondence
is that
\begin{itemize}
\item this graph determines $G$ uniquely, up to conjugacy;
\item this graph is an affine Dynkin diagram of type $ADE$.
\end{itemize}

\subsection{Lie-algebraic setup}\label{s3.2}

Inspired by the classical McKay correspondence, in \cite{MZ1},
we looked into the following problem: let $\mathfrak{g}$
be the Lie algebra $\mathfrak{sl}_2(\mathbb{C})$.
Denote by $\mathscr{D}$ the monoidal category of all finite dimensional 
$\mathfrak{g}$-modules. We note that the monoidal category
$\mathscr{D}$ is monoidally equivalent to the monoidal category
$\mathscr{C}$ from the previous subsection. The category $\mathscr{D}$
acts naturally on the category $\mathfrak{g}\text{-}\mathrm{Mod}$
of all $\mathfrak{g}$-modules.

Furthermore, for any $\mathfrak{g}$-module $N$ of finite length,
we can consider the additive closure $\mathrm{add}(\mathscr{D}\cdot N)$,
inside $\mathfrak{g}\text{-}\mathrm{Mod}$, of all modules of the
form $M\otimes_{\mathbb{C}}N$, where we have $M\in \mathscr{D}$. Then
the action of $\mathscr{D}$ on $\mathfrak{g}\text{-}\mathrm{Mod}$
restricts to the action of $\mathscr{D}$ on 
$\mathrm{add}(\mathscr{D}\cdot N)$.  If $N=0$, then
$\mathrm{add}(\mathscr{D}\cdot N)=0$, so this case is not interesting.
Therefore we assume $N\neq 0$.

In the case $N\neq 0$, the category $\mathrm{add}(\mathscr{D}\cdot N)$
has countably many pair-wise non-isomorphic indecomposable objects.
Here we emphasize that these indecomposable objects are not 
necessarily simple.
Let $X_1,X_2,\dots$ be a complete and irredundant list of 
representatives of the isomorphism classes of indecomposable objects
in $\mathrm{add}(\mathscr{D}\cdot N)$. The category $\mathscr{D}$
is generated, as a monoidal category, by the natural
(2-dimensional) simple $\mathfrak{g}$-module $V=\mathbb{C}^2$.
For all positive integers $i,j$, we can consider the multiplicity
$[V\otimes_{\mathbb{C}}X_i:X_j]$ of the indecomposable module 
$X_j$ as a direct summand  of the (in general, decomposable)
module
$V\otimes_{\mathbb{C}}X_i$. An important difference with the group
case is that we can no longer expect that 
\begin{displaymath}
[V\otimes_{\mathbb{C}}X_i:X_j]= [V\otimes_{\mathbb{C}}X_j:X_i],
\end{displaymath}
in general. The module $V$ is still self-dual, so we do have that
\begin{displaymath}
\dim\mathrm{Hom}_\mathfrak{g}(V\otimes_{\mathbb{C}}X_i,X_j)=
\dim\mathrm{Hom}_\mathfrak{g}(X_i,V\otimes_{\mathbb{C}}X_j),
\end{displaymath}
however, in the case when $X_i$ and/or $X_j$ are not simple,
we can no longer interpret the dimensions of these homomorphism spaces
as multiplicities, in general.

Therefore our infinite multiplicity matrix can be encoded as
a directed graph: the vertices are the $X_i$'s and the number
of arrows from $X_i$ to $X_j$ equals $[V\otimes_{\mathbb{C}}X_i:X_j]$.
It is convenient to simplify this graph and replace each pair 
of opposite arrows by an unoriented arrow. This can, potentially,
create a mixed graph where we have both oriented and unoriented
arrow. For example, the graph
$\xymatrix{\bullet\ar@/^2mm/@{-}[rr]\ar@/_2mm/[rr]&&\bullet}$ 
is the simplification of
$\xymatrix{\bullet\ar[rr]\ar@/_2mm/[rr]&&\bullet\ar@/_2mm/[ll]}$.
We will denote this graph by $\Gamma_N$ and call it the {\em action graph}.

We will be interested in such properties of $\Gamma_N$ as 
being connected or strongly connected and will discuss strongly
connected components of this graph. These notions refer to the original
oriented graph. Strongly connected components of $\Gamma_N$ are
important as they correspond to {\em transitive $\mathscr{D}$-actions} in
the sense of \cite{MM5}. Among transitive $\mathscr{D}$-actions,
of special interest are the so-called {\em simple transitive} actions,
that is, those which do not have any non-trivial $\mathscr{D}$-invariant
ideals. For such $\mathscr{D}$-actions there is a weak form of the
Jordan-H{\"o}lder theory, see \cite{MM5}.

\subsection{First combinatorial result}\label{s3.3}

The following result summarizes \cite[Theorem~24]{MZ1}.

\begin{theorem}\label{thm-s3.3}
Let $L$ be a simple $\mathfrak{sl}_2$-module. Then every strongly connected
component of $\Gamma_L$ is an infinite Dynkin diagram of type
$A_\infty$, $A_\infty^\infty$, $C_\infty$ or $T_\infty$.
\end{theorem}

Here, the type $A_\infty$ is realizable when considering the regular 
action of $\mathscr{D}$ on itself, so we can, for example, choose
$L$ to be the trivial $\mathfrak{sl}_2$-module. 

The type $A_\infty^\infty$ admits many pairwise non-equivalent realizations.
For example, one can choose $L$ to be the simple highest weight module
$L(\lambda)$, for any non-integral highest weight $\lambda$. Using some
classical results of Dixmier on non-isomorphism of primitive quotients of 
$U(\mathfrak{sl}_2)$, see \cite{Di0}, one can show that the realizations via
$L(\lambda)$ and $L(\mu)$, where both $\lambda$ and $\mu$ are non-integral,
are not equivalent, as $\mathscr{D}$-module categories, provided that
$\lambda-\mu$ is not an integer and $(\lambda+1)^2\neq(\mu+1)^2$,
which means that the central characters of $L(\lambda)$ and $L(\mu)$
are different.

To realize the type $C_\infty$, one can take $L$ to be $L(-1)$,
that is, the unique singular integral highest weight module. 
The category $\mathrm{add}(\mathscr{D}\cdot L(-1))$ in this case
will be the category of projective-injective objects in the integral
part of the Bernstein-Gelfand-Gelfand category $\mathcal{O}$,
see \cite{BGG,Hu,Maz1}.

Finally, the type $T_\infty$ is realizable by taking $L$
to be a simple Whittaker module with Whittaker eigenvalue $\frac{1}{2}$.

We note that, in all the examples above, already the action graph
$\Gamma_L$ is strongly connected. For some other choices of $L$,
for example, if one takes $L=L(\lambda)$, for $\lambda\in\{-2,-3,\dots\}$,
the action graph $\Gamma_L$ will not be connected. In this particular
case, the graph $\Gamma_L$ will have two connected components, one of type
$A_\infty$ and the other one of type $C_\infty$.

\subsection{Extending the setup: subalgebras}\label{s3.4}

Similarly to the classical McKay correspondence, one also has the natural
question of how the monoidal category $\mathscr{D}$ acts on the
categories of finite dimensional modules over  Lie subalgebras of
$\mathfrak{g}$. Unlike the classical McKay correspondence, the 
categories of finite dimensional modules over proper Lie subalgebras
of $\mathfrak{g}$ are quite big, in fact, they are significantly bigger
than $\mathscr{D}$ itself. Up to inner automorphisms, there are only a few 
cases to consider, with the case of the zero subalgebra being trivial.
So, let us look at all the cases separately.

In the case of the algebra $\mathfrak{g}$ itself considered as  a
subalgebra of $\mathfrak{g}$, we just get the left regular action of
$\mathscr{D}$ on itself. This has combinatorics of type $A_\infty$.

In the codimension one case, we have a unique, up to an inner automorphism, 
subalgebra, say, the standard Borel subalgebra  $\mathfrak{b}$ of $\mathfrak{g}$.
Let $e,h,f$ be the standard basis of $\mathfrak{g}$. Then $e$ and $h$ form a 
basis of $\mathfrak{b}$. All simple finite dimensional $\mathfrak{b}$-modules 
have dimension $1$ and correspond to weights $\lambda\in \mathbb{C}$.
This means that $h$ acts on the module, which we denote by $N_\lambda$,
by $\lambda$ and $e$ acts as $0$. The category of finite dimensional 
$\mathfrak{b}$-modules is very far from being semi-simple. For example,
all simple modules admit self-extensions (using the action of $h$).
Also, the action of $e$ can be used to produce extensions of 
$N_\lambda$ by $N_{\lambda+2}$. In fact, as is shown in \cite{Mak,OS},
the category of finite dimensional $\mathfrak{b}$-modules has wild
representation type. Consequently, we do not know classification of 
indecomposable objects in this category which makes the action of 
$\mathscr{D}$ on it really difficult to study. Let us again instead
look at the $\mathscr{D}$-module categories of the form
$\mathrm{add}(\mathscr{D}\cdot L)$, where $L$ is a simple
$\mathfrak{b}$-module.

Denote by $M_\lambda$ the induced module 
$U(\mathfrak{b})\otimes_{U(\langle h\rangle)}\mathbb{C}_\lambda$, where 
$\lambda$ is the one-di\-men\-si\-o\-nal $U(\langle h\rangle)$-module on which
$h$ acts as $\lambda$. Then $M_\lambda$ is indecomposable with simple top
$N_\lambda$. Moreover, the kernel of the projection $M_\lambda\tto N_\lambda$
is isomorphic to $M_{\lambda+2}$. Consequently, $M_\lambda$ has an 
infinite composition series given by its radical filtration and the
corresponding simple subquotients are $N_\lambda$, $N_{\lambda+2}$,
$N_{\lambda+4}$ and so on. Denote by $Q(\lambda,k)$ the quotient of 
$M_\lambda$ by $M_{\lambda+2k}$. In particular $N_\lambda=Q(\lambda,1)$.
The following is \cite[Proposition~15]{MZ1}.

\begin{proposition}\label{prop-s3.4-1}
For $\lambda\in\mathbb{C}$, the $\mathscr{D}$-module category
$\mathrm{add}(\mathscr{D}\cdot N_\lambda)$ is simple transitive of
type $A_\infty$. Its indecomposable objects are $Q(\lambda-2(k-1),k)$,
for $k\in\mathbb{Z}_+$, and, up to equivalence of 
$\mathscr{D}$-module categories, it does not depend on $\lambda$.
\end{proposition}

In the codimension two case, up to an inner automorphism, we may
choose the subalgebra to be either $\langle h\rangle$
or $\langle e\rangle$. In both cases, the universal enveloping algebra
of the subalgebra is just the polynomial algebra in one variable.
Therefore the indecomposable objects are classified by the Jordan normal form.
The simple objects are classified by $\lambda\in \mathbb{C}$ and 
have dimension $1$. We denote by $K_\lambda$ the simple 
$\langle h\rangle$-module on which $h$ acts as $\lambda$
and we denote by $F_\lambda$ the simple 
$\langle e\rangle$-module on which $e$ acts as $\lambda$.
For $k\in\mathbb{Z}_{>0}$, we also denote by $F_\lambda^k$
the uniserial $k$-dimensional indecomposable module 
which corresponds to $\lambda$ (it is given by the Jordan
cell of size $k\times k$ with eigenvalue $\lambda$).
The following result combines \cite[Propositions~14 and 15]{MZ1}.

\begin{proposition}\label{prop-s3.4-2}
Let $\lambda\in\mathbb{C}$.

\begin{enumerate}[$($a$)$]
\item\label{prop-s3.4-2.1} 
The $\mathscr{D}$-module category
$\mathrm{add}(\mathscr{D}\cdot K_\lambda)$ is simple transitive of
type $A_\infty^\infty$. Its indecomposable objects are $K_{\lambda+2k}$,
for $k\in\mathbb{Z}$.
\item\label{prop-s3.4-2.2} 
The $\mathscr{D}$-module category
$\mathrm{add}(\mathscr{D}\cdot F_\lambda)$ is transitive of
type $A_\infty$. Its indecomposable objects are $F_\lambda^k$,
for $k\in\mathbb{Z}_+$.
\end{enumerate}
\end{proposition}

In particular, we see that the actions of $\mathscr{D}$ on 
$\mathrm{add}(\mathscr{D}\cdot K_\lambda)$
and $\mathrm{add}(\mathscr{D}\cdot F_\lambda)$
are significantly different.

\subsection{Extending the setup further: Lie algebras for which 
$\mathfrak{sl}_2$ is the Levi factor}\label{s3.5}

Comparing Theorem~\ref{thm-s3.3} and Propositions~\ref{prop-s3.4-1}
and \ref{prop-s3.4-2} with the list of infinite Dynkin diagrams
in Subsection~\ref{s2.4}, we see that the types $B_\infty$ and
$D_\infty$ are missing. To incorporate the type $D_\infty$, we need
to extend our setup as follows: let $\mathfrak{q}$ be a finite
dimensional Lie algebra for which $\mathfrak{g}$ is the Levi quotient.
Then $\mathscr{D}$ can be considered as a category of 
$\mathfrak{q}$-modules through the pullback via the quotient map 
$\mathfrak{q}\tto \mathfrak{g}$. In particular, for any simple
$\mathfrak{q}$-module $N$, the category $\mathrm{add}(\mathscr{D}\cdot N)$
is, naturally, a $\mathscr{D}$-module category.

Now let $\mathfrak{q}$ be the semi-direct product $\mathfrak{g}\ltimes V$,
where $V$ is an abelian ideal given by a simple $5$-dimensional 
$\mathfrak{g}$-module. The following is \cite[Proposition~19]{MZ1}:

\begin{proposition}\label{prop-s3.5-1}
There exists a simple $\mathfrak{q}$-module $N$ such that 
$\mathrm{add}(\mathscr{D}\cdot N)$ is a simple transitive 
$\mathscr{D}$-module category of type $D_\infty$.
\end{proposition}

The construction of the $\mathfrak{q}$-module $N$ is taken from 
\cite{MMr,MMr2}.
It has the property that the action of $\mathfrak{g}$ on it is locally
finite and has finite multiplicities. Additionally, certain generators of
the center of  $U(\mathfrak{q})$ act on $N$ in a very particular
(polynomially related) way.

\subsection{Additional results}\label{s3.6}

The paper \cite{MZ1} contains a number of interesting observations 
about $\mathscr{D}$-module categories that appear in the contexts 
described above. Here are some examples. The following result, which 
is \cite[Proposition~7]{MZ1}, describes a very strong representation
theoretic property of the combinatorial type $A_\infty$
(the term admissible is defined in Subsection~\ref{s5.4}).

\begin{proposition}\label{prop-s3.6-1}
All admissible simple transitive  $\mathscr{D}$-module categories of 
type $A_\infty$ are equivalent to the left regular $\mathscr{D}$-module category.
\end{proposition}

The following result, which 
is \cite[Proposition~20]{MZ1}, provides an interesting representation
theoretic property of the combinatorial type $D_\infty$, even if this
result is not as strong as the uniqueness result in type $A_\infty$ presented above.

\begin{proposition}\label{prop-s3.6-2}
The underlying category of any admissible simple transitive  $\mathscr{D}$-module 
category of type $D_\infty$  is semi-simple.
\end{proposition}

In  type $C_\infty$, the underlying category of the simple 
transitive  $\mathscr{D}$-module category $\mathrm{add}(\mathscr{D}\cdot L(-1))$,
mentioned in Subsection~\ref{s3.3}, is not semi-simple.

In  type $A_\infty^\infty$, a very interesting example appears
in the setup described in Subsection~\ref{s3.4}. We consider 
$\mathfrak{g}$ and its Borel subalgebra $\mathfrak{b}$. Recall the 
$\mathfrak{b}$-modules $M_\lambda$, where $\lambda\in\mathbb{C}$.
Fix $\lambda$ and let $\mathcal{N}_\lambda$ denote the additive closure of
all $M_{\lambda+i}$, where $i\in \mathbb{Z}$. Then 
$\mathcal{N}_\lambda$ is stable under the action of $\mathscr{D}$. In fact,
we have:

\begin{proposition}\label{prop-s3.6-3}
For any $\lambda\in\mathbb{C}$, the $\mathscr{D}$-module 
category $\mathcal{N}_\lambda$ is simple transitive of type $A_\infty^\infty$.
\end{proposition}

The underlying category of $\mathcal{N}_\lambda$ is not semi-simple.
We can consider an abelianization $\overline{\mathcal{N}_\lambda}$ 
of this category (see Subsection~\ref{s5.4} for details)
which is also, naturally, a $\mathscr{D}$-module category.
The original $\mathcal{N}_\lambda$ is a subcategory of 
$\overline{\mathcal{N}_\lambda}$, in fact, it is equivalent to the
category of projective objects in $\mathcal{N}_\lambda$.
Notably, all non-zero objects of $\mathcal{N}_\lambda$ have infinite length.
It turns out that the full subcategory of $\overline{\mathcal{N}_\lambda}$
consisting of objects of finite length is invariant under the action of 
$\mathscr{D}$. Consequently, applying $\mathscr{D}$ to simple objects in
$\overline{\mathcal{N}_\lambda}$ will never output a non-zero projective
object. This seems to be the first example of such a phenomenon
(compare with \cite[Theorem~2]{KMMZ}).

Our final interesting observation in \cite{MZ1} is the following observation,
which is \cite[Proposition~23]{MZ1}, about the type $B_\infty$.
It says that this type is not realizable in our setups.

\begin{proposition}\label{prop-s3.6-4}
Simple transitive $\mathscr{D}$-module categories of
type $B_\infty$ over the complex numbers whose underlying
category is locally finitary and has weak kernels
do not exist.
\end{proposition}

The type $B_\infty$ combinatorics is, by definition, dual to the type
$C_\infty$ combinatorics. Therefore one can find the type $B_\infty$ combinatorics
by considering the basis of simple (instead of projective) modules
in the examples which  realize the type $C_\infty$ combinatorics.

\section{$\mathfrak{sl}_3$-combinatorics}\label{s4}

\subsection{Setup}\label{s4.1}

In \cite{MZ2}, we tried to generalize (some of) the results of 
\cite{MZ1} to the case of the Lie algebra $\mathfrak{sl}_3(\mathbb{C})$.
So, in this section, we let $\mathfrak{g}=\mathfrak{sl}_3$.
We denote by $\mathscr{B}$ the monoidal category of finite
dimensional $\mathfrak{g}$-modules with the usual monoidal structure.

The category $\mathscr{B}$ is generated, in a weak sense,
by the natural $\mathfrak{g}$-module $V:=\mathbb{C}^3$. Here,
by a weak sense, we mean that any indecomposable object in
$\mathscr{B}$ is a summand of some tensor power of $V$.
If we, additionally, consider the dual $\mathfrak{g}$-module $V^*$,
then $\mathscr{B}$ is generated, as a monoidal category, by $V$
and $V^*$ in the following, much stronger, sense: we can enumerate
indecomposable objects of $\mathscr{B}$ by positive integers,
say $B_1,B_2,\dots$ such that, for each $i\in\mathbb{Z}_{>0}$, 
there exist $a,b\in\mathbb{Z}_{\geq 0}$ with the property that
$V^{\otimes a}\otimes (V^*)^{\otimes b}$ has $B_i$ as a summand
with multiplicity  $1$ and all other summands are isomorphic
to $B_j$, for $j<i$. We also note that, given a monoidal action of
$\mathscr{B}$ on any $\mathscr{B}$-module category, the objects
$V$ and $V^*$ necessarily act as biadjoint functors. 

Given a simple $\mathfrak{g}$-module $L$, the category 
$\mathrm{add}(\mathscr{B}\cdot L)$ is an idempotent split 
Krull-Schmidt category with countably many indecomposable objects
and finite dimensional morphism spaces. The category
$\mathrm{add}(\mathscr{B}\cdot L)$ has the obvious
structure of a $\mathscr{B}$-module category. Let
$X_1,X_2,\dots$ be a complete and irredundant list of 
representatives of the isomorphism classes of indecomposable
objects in the category $\mathrm{add}(\mathscr{B}\cdot L)$. We consider two
oriented graphs, $\Gamma_L$ and $\Gamma^*_L$. For both of them,
the set of vertices is in bijection with the $X_i$'s. For
$\Gamma_L$, the number of oriented edges from $X_i$ to $X_j$
equals the multiplicity of $X_j$ as a summand of $V\otimes_{\mathbb{C}}X_i$.
Similarly, for $\Gamma^*_L$, the number of oriented edges from $X_i$ to $X_j$
equals the multiplicity of $X_j$ as a summand of $V^*\otimes_{\mathbb{C}}X_i$.
From the previous paragraph, it follows that the graphs 
$\Gamma_L$ and $\Gamma^*_L$ completely determine the combinatorics
of the action on $\mathscr{B}$ on $\mathrm{add}(\mathscr{B}\cdot L)$
in the sense that, for every object $B\in \mathscr{B}$ and all
$i,j$, the multiplicity of $X_j$ in $B\otimes_{\mathbb{C}}X_i$
is uniquely determined. Consequently, the problem to classify
all possible $\Gamma_L$ and $\Gamma^*_L$ is natural and interesting.

\subsection{Main results}\label{s4.2}

The main result of \cite{MZ2} is \cite[Theorem~20]{MZ2}, which can be formulated
as follows:

\begin{theorem}\label{thm-s4.2-1}
For a simple $\mathfrak{g}$-module $L$, any strongly connected component of 
the graph $\Gamma_L$ is isomorphic to one of the graphs in Figure~\ref{fig1},
with the graphs in Figure~\ref{fig2} describing the corresponding
strongly connected component of $\Gamma^*_L$.
\end{theorem}

\begin{figure}
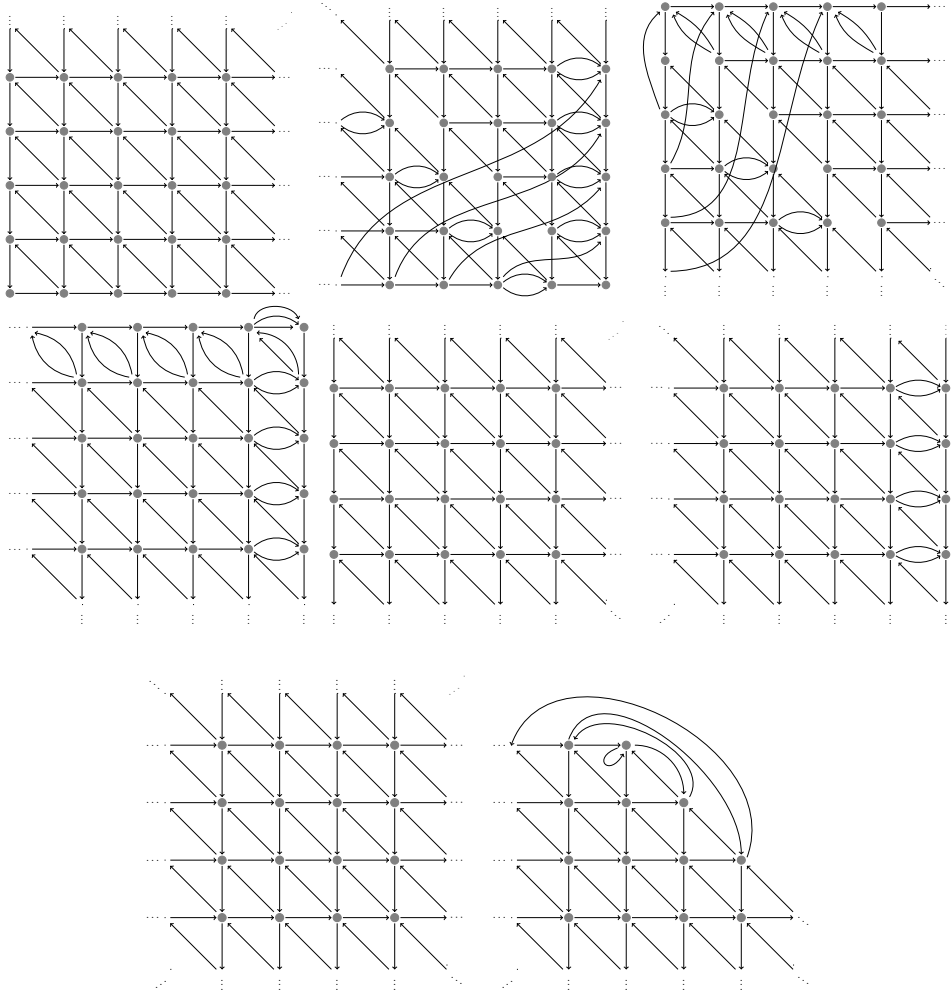

\resizebox{\textwidth}{!}{

}
\caption{The graphs $\Gamma_L$}\label{fig1}
\end{figure}

\begin{figure}
\resizebox{\textwidth}{!}{
%
}
\caption{The graphs $\Gamma^*_L$}\label{fig2}
\end{figure}

\subsection{Additional results}\label{s4.3}

Similarly to the $\mathfrak{sl}_2$-case, some combinatorial patterns 
provide additional rep\-re\-sen\-ta\-ti\-on-theoretic information.
The following is \cite[Theorem~1]{MZ2} and is an analogue of 
Proposition~\ref{prop-s3.6-1}.

\begin{proposition}\label{prop-s4.3-1}
Let $\mathcal{M}$ be a simple transitive 
admissible $\mathscr{B}$-module category
whose combinatorics of the action of $V$ is given by the 
first graph in Figure~\ref{fig1} and, respectively, 
by the first graph in Figure~\ref{fig2}, for $V^*$.
Then $\mathcal{M}$ is equivalent to the left regular 
$\mathscr{B}$-module category ${}_\mathscr{B}\mathscr{B}$.
\end{proposition}

\section{Other Lie algebras}\label{s5}

\subsection{Setup}\label{s5.1}

Let now $\mathfrak{g}$ be an arbitrary semi-simple finite dimensional
complex Lie algebra. Consider the monoidal category $\mathscr{F}$ of all
finite dimensional $\mathfrak{g}$-modules in which the monoidal structure is the usual
one, given by tensoring over $\mathbb{C}$ and using the usual 
comultiplication for $U(\mathfrak{g})$. The monoidal unit is the
trivial $\mathfrak{g}$-module. We note that, for any $V\in \mathscr{F}$,
the dual object $V^*\in \mathscr{F}$ is biadjoint to $V$ in $\mathscr{F}$
(sometimes referred to as a dual object in the monoidal sense).
Note that $\mathscr{F}$ is a semi-simple category.

For any simple $\mathfrak{g}$-module $L$ and any $V,V'\in \mathscr{F}$,
we have 
\begin{displaymath}
\dim\mathrm{Hom}_\mathfrak{g}(V\otimes_\mathbb{C} L,V'\otimes_\mathbb{C}L)
<\infty.
\end{displaymath}
Consequently, the additive closure $\mathrm{add}(\mathscr{F}\cdot L)$
is an idempotent split Krull-Schmidt category with finite dimensional 
morphism spaces and countably many indecomposable objects. The category
$\mathrm{add}(\mathscr{F}\cdot L)$ has the natural structure of an
$\mathscr{F}$-module category. It is a natural (but probably very difficult) 
problem to classify, up to equivalence, simple transitive
subquotients of all possible $\mathscr{F}$-module categories of the 
form $\mathrm{add}(\mathscr{F}\cdot L)$. Here we remark 
that we know that already for $\mathfrak{g}=\mathfrak{sl}_2$ there are
infinitely (even uncountably) many such categories. As a first step
towards this very difficult problem, it is natural to understand the
combinatorics of such categories. For the cases
$\mathfrak{g}=\mathfrak{sl}_2$ and $\mathfrak{g}=\mathfrak{sl}_3$,
this is described in Sections~\ref{s3} and \ref{s4}, respectively.

\subsection{Combinatorial setup}\label{s5.2}

Fix a triangular decomposition 
\begin{displaymath}
\mathfrak{g}= \mathfrak{n}_-\oplus \mathfrak{h}\oplus\mathfrak{n}_+
\end{displaymath}
of $\mathfrak{g}$. Here $\mathfrak{h}$ is a Cartan subalgebra and
$\mathfrak{n}_+$ and $\mathfrak{n}_-$ are the Lie subalgebras
corresponding to a fixed splitting of all roots of $\mathfrak{g}$
into positive and negative roots, respectively.

Consider the Grothendieck ring $\mathbf{Gr}(\mathscr{F})$ of
$\mathscr{F}$. Since $\mathscr{F}$ is symmetric, 
the ring $\mathbf{Gr}(\mathscr{F})$ is commutative.
Let $n$ be the rank of $\mathfrak{g}$ and
$\varpi_1,\varpi_2,\dots,\varpi_n$ be the fundamental
weights. Then simple objects of $\mathscr{F}$ are in bijection
with the elements in the $\mathbb{Z}_{\geq 0}$-linear span
of the fundamental weights. We denote this span by 
$\mathfrak{h}^*_\mathrm{idom}$. For each $\lambda\in \mathfrak{h}^*_\mathrm{idom}$,
we denote by $L(\lambda)$ the corresponding simple object in
$\mathscr{F}$ which is the simple highest weight module
(with respect to our choice of the triangular decomposition)
with highest weight $\lambda$.

For $\displaystyle\lambda=\sum_{i=1}^n k_i \varpi_i$ and 
$\displaystyle\mu=\sum_{i=1}^n m_i \varpi_i$, we write $\lambda\leq \mu$
provided that $k_i\leq m_i$, for all $i$.
If $\displaystyle\lambda=\sum_{i=1}^n k_i \varpi_i\in \mathfrak{h}^*_\mathrm{idom}$, then
the object
\begin{equation}\label{eq-s5.2-1}
\bigotimes_{i=1}^nL(\varpi_i)^{\otimes k_i} 
\end{equation}
has a unique summand isomorphic to $L(\lambda)$ and all other summands
are of the form $L(\mu)$, for $\mu<\lambda$. Therefore,
there is a ring isomorphism between $\mathbf{Gr}(\mathscr{F})$
and the polynomial ring $\mathbb{Z}[x_1,x_2,\dots,x_n]$ which sends 
the object of $\mathscr{F}$ given by Formula~\eqref{eq-s5.2-1}  
to $x_1^{k_1}x_2^{k_2}\dots x_n^{k_n}$.
We identify $\mathbf{Gr}(\mathscr{F})$
with $\mathbb{Z}[x_1,x_2,\dots,x_n]$ via this isomorphism.

Now let $L$ be a simple $\mathfrak{g}$-module. 
We consider the split Grothendieck group 
$[\mathrm{add}(\mathscr{F}\cdot L)]_\oplus$.
The action of $\mathscr{F}$ on $\mathrm{add}(\mathscr{F}\cdot L)$
makes $[\mathrm{add}(\mathscr{F}\cdot L)]_\oplus$ into
a $\mathbf{Gr}(\mathscr{F})$-module. The abelian group 
$[\mathrm{add}(\mathscr{F}\cdot L)]_\oplus$ is countably generated.
The $\mathbf{Gr}(\mathscr{F})$-module structure on
$[\mathrm{add}(\mathscr{F}\cdot L)]_\oplus$ is
uniquely determined by the action of the generators
$x_1,x_2,\dots,x_n$ on $[\mathrm{add}(\mathscr{F}\cdot L)]_\oplus$.

Let $P_1,P_2,\dots$ be a complete and irredundant list of 
representatives of the isomorphism classes of indecomposable
objects in $\mathrm{add}(\mathscr{F}\cdot L)$.
The action of each $x_i$ on $[\mathrm{add}(\mathscr{F}\cdot L)]_\oplus$
is given by an infinite matrix $[x_i]$, whose rows and columns are
indexed by the $P_j$'s. The entry in the intersection of 
row $j$ and column $j'$ equals the multiplicity of $P_j$
as a summand of $L(\varpi_i)\otimes_\mathbb{C}P_{j'}$.
In particular, each such entry is a non-negative integer.
Hence the matrix $[x_i]$ can be represented by a 
directed graph, whose vertices are the $P_j$'s and the number
of oriented edges from $P_{j'}$ to $P_j$ equals the
multiplicity of $P_j$
as a summand of $L(\varpi_i)\otimes_\mathbb{C}P_{j'}$.
All this is a straightforward generalization of the 
special cases the Lie algebras $\mathfrak{sl}_2$
and $\mathfrak{sl}_3$ discussed in the previous sections.

\subsection{Regular actions}\label{s5.4}

Our first observation is analogous to 
Propositions~\ref{prop-s3.6-1} and \ref{prop-s4.3-1}.
Recall that an $\mathscr{F}$-module category
$\mathcal{M}$ is {\em admissible} provided that 
it is idempotent split, has finite dimensional morphism
spaces and weak kernels. If $\mathcal{M}$ is admissible,
then the abelianization $\overline{\mathcal{M}}$ of 
$\mathcal{M}$ is defined as a category whose objects are 
diagrams $X\to Y$ over $\mathcal{M}$ and morphisms are
equivalence classes of solid commutative diagrams
\begin{displaymath}
\xymatrix{
X\ar[rr]\ar[d]&&Y\ar@{.>}[dll]_b\ar[d]^{a}\\
X'\ar[rr]_c&&Y'\\
}
\end{displaymath}
where the equivalence is generated by the relation that
a solid diagram is equivalent to $0$ provided that 
the morphism $a$ can be factorized as $cb$, for some $b$.
The category $\overline{\mathcal{M}}$ is an 
$\mathscr{F}$-module category via the component-wise action.

\begin{theorem}\label{thm-s5.4-1}
Let $\mathcal{M}$ be an admissible simple transitive $\mathscr{F}$-module category
such that $[\mathcal{M}]_\oplus$ is isomorphic to $[\mathscr{F}]_\oplus$
as an $\mathbf{Gr}(\mathscr{F})$-module. Then
$\mathcal{M}$ is equivalent to ${}_\mathscr{F}\mathscr{F}$
as an $\mathscr{F}$-module category.
\end{theorem}

\begin{proof}
The category $\mathscr{F}$ contains a distinguished indecomposable object,
namely, the identity object. Let $I\in \mathcal{M}$ be the indecomposable
object of $\mathcal{M}$ that corresponds to this identity object under the
isomorphism between $[\mathcal{M}]_\oplus$ and $[\mathscr{F}]_\oplus$,
as $\mathbf{Gr}(\mathscr{F})$-modules.

Then, for any indecomposable object $F\in \mathscr{F}$, the object
$F(I)$ is also indecomposable due to the combination of the facts that 
this is true in $\mathscr{F}$ and that $[\mathcal{M}]_\oplus$ and $[\mathscr{F}]_\oplus$
are isomorphic  as an $\mathbf{Gr}(\mathscr{F})$-module.

Consider the abelianization $\overline{\mathcal{M}}$ of $\mathcal{M}$,
which is, naturally, an $\mathscr{F}$-module category. Let
$\{M_q\,:\,q\in Q\}$ be a complete and irredundant list of 
representatives of the isomorphism classes of indecomposable
objects in $\mathcal{M}$. For each $q$, denote by 
$N_q$ the simple top of $M_q$, considered as an object of 
$\overline{\mathcal{M}}$. Then, for any
$F\in \mathscr{F}$, we have the matrix $[F]$ which records the
direct summand multiplicities $[F(M_q):M_p]$. We also have 
the matrix $\llbracket F\rrbracket$ which records the
composition multiplicities $[F(N_q):N_p]$. 

Since  $\mathscr{F}$ is semi-simple, by \cite[Lemma~8]{AM}
applied to ${}_\mathscr{F}\mathscr{F}$, we have 
that, for any $F\in \mathscr{F}$, the matrix $[F]$ is transposed to 
$[F^*]$. By the same argument applied to $\mathcal{M}$, we have
that $[F]$ is transposed to 
$\llbracket F^*\rrbracket$. Hence $[F]=\llbracket F\rrbracket$.

First, we claim that the action of $\mathscr{F}$ leaves the category
of semi-simple objects in $\overline{\mathcal{M}}$ invariant. 
Let $q_0\in Q$ be such that $M_{q_0}\cong I$.
Each simple in $\overline{\mathcal{M}}$ is of the form
$F(N_{q_0})$, for some $F\in \mathscr{F}$ (this is true because of the
isomorphism between between $[\mathcal{M}]_\oplus$ and 
$[\mathscr{F}]_\oplus$ and the fact that such a claim is obviously true in 
${}_\mathscr{F}\mathscr{F}$). Given $G\in \mathscr{F}$,
we have $G(F(I))\cong (G\circ F)(I)$ which is semi-simple 
since $\mathscr{F}$ is semi-simple. This proves our claim.

Next, we claim that the radical of $\overline{\mathcal{M}}$ is
$\mathscr{F}$-invariant (and hence is zero due to simple
transitivity). Indeed, applying $F$ to the short exact sequence
\begin{displaymath}
0\to \mathrm{Rad}(M_q)\to M_q\to N_q\to 0, 
\end{displaymath}
we get the exact sequence
\begin{displaymath}
0\to F(\mathrm{Rad}(M_q))\to F(M_q)\to F(N_q)\to 0.
\end{displaymath}
By the previous paragraph, the top of $F(M_q)$ is isomorphic
to $F(N_q)$ which implies that $F(\mathrm{Rad}(M_q))$ coincides with
$\mathrm{Rad}(F(M_q))$. This establishes our claim. In particular,
$\overline{\mathcal{M}}\cong \mathcal{M}$, so $\mathcal{M}$
is semi-simple.

Now, consider the Yoneda map from ${}_\mathscr{F}\mathscr{F}$
to $\mathcal{M}$ that sends $\mathbf{1}$ to $I$. It is 
a homomorphism of $\mathscr{F}$-module categories by constructions.
As we already
established above,
it sends indecomposable objects to indecomposable objects.
Since $\overline{\mathcal{M}}\cong \mathcal{M}$, it is 
an equivalence of categories.  This completes the proof.
\end{proof}

\subsection{Projective functors}\label{s5.3}

The action of the monoidal category $\mathscr{F}$ on 
$\mathfrak{g}$-modules is closely related to the notion of
projective functor, introduced in \cite{BG}.

We denote by $\mathcal{Z}$ the full subcategory of 
the category of all finitely generated $\mathfrak{g}$-modules
that consists of all objects on which the center $Z(\mathfrak{g})$
of the universal enveloping algebra $U(\mathfrak{g})$ of
$\mathfrak{g}$ acts locally finitely. Given a central character $\chi$,
denote by $\mathcal{Z}_{{}_\chi}$ the full subcategory of 
$\mathcal{Z}$ consisting of all objects
on which the kernel of $\chi$ acts locally nilpotently.
Then the category $\mathcal{Z}$ decomposes into a direct sum of 
the subcategories $\mathcal{Z}_{{}_\chi}$, taken over all $\chi$. 
The category $\mathcal{Z}$ is invariant under the usual action of 
the monoidal category $\mathscr{F}$ on all $\mathfrak{g}$-modules.

An endofunctor of $\mathcal{Z}$ is called a {\em projective functor} 
provided that it is isomorphic to a direct summand of the functor of
tensoring with some finite dimensional $\mathfrak{g}$-module. This
notion was introduced in \cite{BG}, where indecomposable projective
functors were classified. It turns out that  indecomposable projective 
are in bijection with the orbits of the Weyl group $W$, with respect to the dot-action, 
on pairs $(\lambda,\mu)\in(\mathfrak{h}^*)^2$, where 
$\lambda-\mu\in\Lambda$. Each orbit of this form contains at least 
one pair $(\lambda,\mu)$ with the properties that
\begin{itemize}
\item the weight $\lambda$ is dominant with respect to its integral Weyl group;
\item the weight $\mu$ is anti-dominant with respect to the stabilizer of 
$\lambda$ (for the dot-action).
\end{itemize}
A pair  $(\lambda,\mu)$ satisfying these conditions is called {\em proper}. 
The  indecomposable projective functor corresponding to a pair 
$(\lambda,\mu)$ as above is denoted $\theta_{\lambda,\mu}$.

The relevance of projective functors in our case stems from the 
classical property that any simple $\mathfrak{g}$-module has a central
character, see \cite[Proposition~2.6.8]{Di}. Consequently, 
for a simple $\mathfrak{g}$-module $L$, the category 
$\mathrm{add}(\mathscr{F}\cdot L)$ is a subcategory of 
$\mathcal{Z}$ and hence the action of $\mathscr{F}$ on 
$\mathrm{add}(\mathscr{F}\cdot L)$ can be studied using projective functors.

\subsection{Generic blocks}\label{s5.5}

Let $\lambda\in\mathfrak{h}^*$ be a weight. For this $\lambda$, we denote by 
$\chi_{{}_\lambda}$ the central character of the Verma module $\Delta(\lambda)$
with highest weight $\lambda$. We also denote by $\Lambda$ the set of all
{\em integral weights}, that is, the set of all weights which appear in 
finite dimensional $\mathfrak{g}$-modules.

We will say that  $\lambda$ is {\em generic}, provided that, for any 
$\mu,\nu\in\Lambda$ with $\mu\neq\nu$, the central characters $\chi_{{}_{\lambda+\mu}}$ 
and $\chi_{{}_{\lambda+\nu}}$ are different. Note that the condition of 
being generic can be described, for elements of $\mathfrak{h}^*$,
as the conjunction of a countable set of polynomial inequalities.
In particular, the Lebesgue measure of the set of all non-generic elements 
equals zero. In this sense, almost all weights are generic.
We will say that a central character is {\em generic} provided that it is 
of the form $\chi_{{}_{\lambda}}$, for a generic $\lambda$.

If $\lambda$ is generic and $\mu,\nu\in\Lambda$, then there
is a unique, up to isomorphism, indecomposable projective functor 
from $\mathcal{Z}_{{}_{\chi_{{}_{\lambda+\mu}}}}$ to 
$\mathcal{Z}_{{}_{\chi_{{}_{\lambda+\nu}}}}$. Namely, this functor is
$\theta_{\lambda+\mu,\lambda+\nu}$ and it is 
an equivalence of categories with inverse
$\theta_{\lambda+\nu,\lambda+\mu}$.

\begin{theorem}\label{thm-s5.5-1}
Let $L$ be a simple $\mathfrak{g}$-module with a generic central 
character $\chi_{{}_{\lambda}}$.
Then we have the following:
\begin{enumerate}[$($a$)$]
\item\label{thm-s5.5-1.1} The category $\mathrm{add}(\mathscr{F}\cdot L)$
is semi-simple. 
\item\label{thm-s5.5-1.2} Up to isomorphism, the simple objects of 
$\mathrm{add}(\mathscr{F}\cdot L)$ are in bijection with 
elements in $\Lambda$: for $\mu\in \Lambda$, the corresponding simple
object is $\theta_{\lambda,\lambda+\mu}(L)$.
\item\label{thm-s5.5-1.3} As an $\mathscr{F}$-module category, the
category $\mathrm{add}(\mathscr{F}\cdot L)$ is simple transitive.
\item\label{thm-s5.5-1.4} The $\mathbf{Gr}(\mathscr{F})$-module
$[\mathrm{add}(\mathscr{F}\cdot L)]_\oplus$ does not depend on
$L$, up to isomorphism, and has the following description: for any $M\in \mathscr{F}$
and $\mu,\nu\in\Lambda$, the multiplicity of 
$\theta_{\lambda,\lambda+\nu}(L)$ as a summand 
(equivalently, subquotient) of 
$M\otimes_{\mathbb{C}}\theta_{\lambda,\lambda+\mu}(L)$
equals $\dim M_{\nu-\mu}$.
\end{enumerate}
\end{theorem}

\begin{proof}
Since any indecomposable projective functor between the blocks of 
$\mathcal{Z}$ corresponding to  generic central characters is
an equivalence, we have that, for any $M\in \mathscr{F}$, the module
$M\otimes_{\mathbb{C}}L$ is semi-simple. This implies Claim~\eqref{thm-s5.5-1.1}.
Claim~\eqref{thm-s5.5-1.2} follows directly from the classification of 
indecomposable projective functor between the blocks of 
$\mathcal{Z}$ corresponding to  generic central characters.

Since we now know that the underlying category of 
$\mathrm{add}(\mathscr{F}\cdot L)$ is semi-simple, to prove its simple
transitivity, as an $\mathscr{F}$-module category, we just need
to prove its transitivity. For this, it is enough to show that,
for any $\mu\in \Lambda$, the module $L$ belongs to 
$\mathrm{add}(\mathscr{F}\cdot \theta_{\lambda,\lambda+\mu}(L))$.
We have $L\cong\theta_{\lambda+\mu,\lambda}(\theta_{\lambda,\lambda+\mu}(L))$
as $\theta_{\lambda+\mu,\lambda}$ is an equivalence inverse to 
$\theta_{\lambda,\lambda+\mu}$. This implies Claim~\eqref{thm-s5.5-1.3}.

To prove Claim~\eqref{thm-s5.5-1.4}, it is enough to show that, for any
$M\in \mathscr{F}$ and $\mu\in\Lambda$, the multiplicity 
of $\theta_{\lambda,\lambda+\mu}$ in the endofunctor $M\otimes_{\mathbb{C}}{}_-$
of $\mathrm{add}(\mathscr{F}\cdot L)$
equals $\dim M_\mu$. This follows directly from
\cite[Corollary~5.5]{Ko1}. This completes the proof.
\end{proof}

\begin{remark}
In the case of $\mathfrak{sl}_2$, the generic 
combinatorics is described by the $A_\infty^\infty$ diagram.
In the case of $\mathfrak{sl}_3$, the generic 
combinatorics is described by the first diagram in the
last row in Figure~\ref{fig1}
(and the corresponding diagram in Figure~\ref{fig2}).
\end{remark}

\subsection{Expectations and further directions}\label{s5.6}

We call a module over $\mathbf{Gr}(\mathscr{F})$ realizable if it is
isomorphic to $[\mathcal{M}]_\oplus$, for some transitive 
subcategory of $\mathrm{add}(\mathscr{F}\cdot L)$, where 
$L$ is a simple $\mathfrak{g}$-module.

\begin{conjecture}\label{conj-s5.6-1}
There are only finitely many  realizable
$\mathbf{Gr}(\mathscr{F})$-modules, up to isomorphism.
\end{conjecture}

We strongly believe that Conjecture~\ref{conj-s5.6-1} is true.
In particular, it is true for $\mathfrak{g}=\mathfrak{sl}_2$
and $\mathfrak{g}=\mathfrak{sl}_3$, as explained above.
However, at the moment we do not see how to prove it.
With or without Conjecture~\ref{conj-s5.6-1}, the problem to
classify all realizable $\mathbf{Gr}(\mathscr{F})$-modules, up
to isomorphism, seems to be very natural and interesting.
This would be the first step towards classification of 
simple $\mathscr{F}$-module categories, up to equivalence.
Looking at our $\mathfrak{sl}_2$-results, we expect the latter
classification be much more difficult than the former.

For a fixed central character $\chi$, the main result of 
\cite{MMMTZ} implies that, up to equivalence, there are only 
finitely many simple transitive finitary module categories
over the monoidal category of projective endofunctors of 
$\mathcal{Z}_\chi$. This, of course, also means that 
there are only finitely many corresponding combinatorial shadows.
However,  Conjecture~\ref{conj-s5.6-1} is about all $\chi$
at the same time and hence the results of \cite{MMMTZ} are
not directly applicable to prove this conjecture.

Study of the categories of the form $\mathrm{add}(\mathscr{F}\cdot L)$
leads to a natural equivalence relation on the set $\mathrm{Irr}(\mathfrak{g})$
of the isomorphism classes of simple $\mathfrak{g}$-modules: 
two simple $\mathfrak{g}$-modules $L$ and $L'$ are said to be {\em equivalent}
provided that 
\begin{displaymath}
\mathrm{add}(\mathscr{F}\cdot L)=
\mathrm{add}(\mathscr{F}\cdot L'),
\end{displaymath}
where $=$ really means the equality, not just
an equivalences of $\mathscr{F}$-module categories. This is closely
related to the partial pre-order $\triangleright$ on $\mathrm{Irr}(\mathfrak{g})$
introduced in \cite{MMM}: we have $L\triangleright L'$ if and only if 
there is a finite dimensional $\mathfrak{g}$-module $F$ such that
the module $F\otimes_\mathbb{C}L$ surjects onto $L'$. Understanding properties of
these (pre-)orders seems to be an essential step in this theory and
has further potential applications. 

The main focus of \cite{MMM} was on understanding the socle of
the modules of the form $F\otimes_\mathbb{C}L$ as above. This would
provide more essential information on the structure of the
indecomposable objects in the category $\mathrm{add}(\mathscr{F}\cdot L)$.
The main results of \cite{MMM} asserts that the socle in question is
a finite length module, under the additional assumption that 
$L$ is holonomic, that is, has minimal possible Gelfand-Kirillov dimension
for its annihilator. It type $A$, a similar results is known for all $L$,
see \cite{CCM}. In both cases, the proofs are heavily based on 
the main result of \cite{MMMTZ} (in type $A$, on the earlier special 
case  of it appearing in \cite{MM5}). This is additional evidence that
understanding combinatorial properties of Lie-theoretic action can be
really helpful for studying purely Lie-theoretic and representation 
theoretic properties of Lie algebra modules.


\begin{thebibliography}{999999999}

\bibitem[AM11]{AM}
Agerholm, T.; Mazorchuk, V.
On selfadjoint functors satisfying polynomial relations.
J. Algebra {\bf 330} (2011), 448--467.

%

\bibitem[Ar72]{Ar}
Arnol'd, V.
Normal forms of functions near degenerate critical points, 
the Weyl groups $A_k$, $D_k$, $E_k$ and Lagrangian singularities.
Funkcional. Anal. i Prilozen. {\bf 6} (1972), no. 4, 3--25.

\bibitem[BG80]{BG}
Bernstein, J.; Gelfand, S.
Tensor products of finite- and infinite-dimensional 
representations of semisimple Lie algebras. 
Compositio Math. {\bf 41} (1980), no.2, 245--285.

\bibitem[BGG76]{BGG}
Bernstein, I.; Gelfand, I.; Gelfand, S.
A certain category of $\mathfrak{g}$-modules.
Funkcional. Anal. i Prilozen. {\bf 10} (1976), no. 2, 1--8.


\bibitem[CIZ87]{CIZ}
Cappelli, A.; Itzykson, C.; Zuber, J.-B.
The A-D-E classification of minimal and $A^{(1)}_1$-conformal 
invariant theories. Comm. Math. Phys. {\bf 113} (1987), no. 1, 1--26.


%

\bibitem[Co08]{Co}
Coecke, B. Introducing categories to the practicing physicist.
Preprint: arXiv:0808.1032.

\bibitem[CP11]{CP}
Coecke, B.; Paquette, {\'E}. O.
Categories for the practising physicist.
Lecture Notes in Phys., {\bf 813}
Springer, Heidelberg, 2011, 173--286.

%

\bibitem[CCM21]{CCM}
Chen, C.-W.; Coulembier, K.; Mazorchuk, V.
Translated simple modules for Lie algebras and 
simple supermodules for Lie superalgebras.
Math. Z. {\bf 297} (2021), no. 1-2, 255--281.

\bibitem[Di73]{Di0}
Dixmier, J.
Quotients simples de l'alg{\`e}bre enveloppante de 
$\mathfrak{sl}(2)$. 
J. Algebra {\bf 24} (1973), 551--564.

\bibitem[Di96]{Di}
Dixmier, J. Enveloping algebras.
Grad. Stud. Math., {\bf 11}
American Mathematical Society, Providence, RI, 1996, xx+379 pp.

%
 
\bibitem[EW06]{EW}
Erdmann, K.; Wildon, M. Introduction to Lie algebras.
Springer Undergrad. Math. Ser.,
Springer-Verlag London, Ltd., London, 2006, x+251 pp.
 
\bibitem[E-O15]{EGNO}
Etingof, P.; Gelaki, S.; Nikshych, D.; Ostrik, V.
Tensor categories.
Math. Surveys Monogr., {\bf 205}
American Mathematical Society, Providence, RI, 2015, xvi+343 pp.
 
%
%
%

\bibitem[Ga72]{Ga}
Gabriel, P. Unzerlegbare Darstellungen. I.
Manuscripta Math. {\bf 6} (1972), 71--103. 


\bibitem[HPR80a]{HPR}
Happel, D.; Preiser, U.; Ringel, C.~M.
Binary polyhedral groups and Euclidean diagrams.
Manuscripta Math. {\bf 31} (1980), no. 1-3, 317--329.
 
\bibitem[HPR80b]{HPR2}
Happel, D.; Preiser, U.; Ringel, C.~M.
Vinberg's characterization of Dynkin diagrams using 
subadditive functions with application to DTr-periodic modules.
Lecture Notes in Math., {\bf 832},
Springer, Berlin, 1980, pp. 280--294.


\bibitem[Hu75]{Hu75}
Humphreys, J. Linear algebraic groups.
Grad. Texts in Math., No. {\bf 21},
Springer-Verlag, New York-Heidelberg, 1975, xiv+247 pp.

\bibitem[Hu08]{Hu}
Humphreys, J.
Representations of semisimple Lie algebras in the BGG category $\mathcal{O}$.
Grad. Stud. Math., {\bf 94}
American Mathematical Society, Providence, RI, 2008, xvi+289 pp.

%

\bibitem[K-Z19]{KMMZ}
Kildetoft, T.; Mackaay, M.; Mazorchuk, V.; Zimmermann, J.
Simple transitive 2-representations of small 
quotients of Soergel bimodules.
Trans. Amer. Math. Soc. {\bf 371} (2019), no. 8, 5551--5590.


\bibitem[KO02]{KO}
Kirillov, A., Jr.; Ostrik, V.
On a $q$-analogue of the McKay correspondence and the ADE 
classification of $\mathfrak{sl}_2$ conformal field theories.
Adv. Math. {\bf 171} (2002), no. 2, 183--227.


\bibitem[Kn02]{Kn}
Knapp, A. Lie groups beyond an introduction.
Progr. Math., {\bf 140}
Birkh{\"a}user Boston, Inc., Boston, MA, 2002, xviii+812 pp.


\bibitem[Ko75]{Ko1}
Kostant, B.
On the tensor product of a finite and an infinite dimensional representation.
J. Functional Analysis {\bf 20} (1975), no. 4, 257--285.



\bibitem[MMM24]{MMM}
Mackaay, M.; Mazorchuk, V.; Miemietz, V.
Applying projective functors to arbitrary holonomic simple modules.
J. Lond. Math. Soc. (2) {\bf 110} (2024), no. 2, Paper No. e12965, 29 pp.

\bibitem[M-Z23]{MMMTZ}
Mackaay, M.; Mazorchuk, V.; 
Miemietz, V.; Tubbenhauer, D.; Zhang, X.
Simple transitive $2$-representations of Soergel bimodules for 
finite Coxeter types. 
Proc. Lond. Math. Soc. (3) {\bf 126} (2023), no. 5, 1585--1655.


\bibitem[MT19]{MT}
Mackaay, M., Tubbenhauer, D.
Two-color Soergel calculus and simple transitive 2-representations.
Canad. J. Math. {\bf 71} (2019), no. 6, 1523--1566.

%
%
%

\bibitem[Mak12]{Mak} 
Makedonskyi, Ie.  On wild Lie algebras. Preprint arXiv:1202.1401.

\bibitem[Ma10]{Maz1} Mazorchuk, V. 
Lectures on $\mathfrak{sl}_2(\mathbb{C})$-modules.
Imperial College Press, London, 2010, x+263~pp.


\bibitem[MM16]{MM5} 
Mazorchuk, V.; Miemietz, V.
Transitive 2-representations of finitary 2-categories.
Trans. Amer. Math. Soc. {\bf 368} (2016), no. 11, 7623--7644.

\bibitem[MMr22a]{MMr} 
Mazorchuk, V.; Mr{\dj}en, R.
Lie algebra modules which are locally finite and with finite 
multiplicities over the semisimple part.
Nagoya Math. J. {\bf 246} (2022), 430--470.

\bibitem[MMr22b]{MMr2} 
Mazorchuk, V.; Mr{\dj}en, R.
$\mathfrak{sl}_2$-Harish-Chandra modules for 
$\mathfrak{sl}_2\ltimes L(4)$.
J. Math. Phys. {\bf 63} (2022), no. 2, Paper No. 
021701, 21 pp.


\bibitem[MZ24]{MZ1}
Mazorchuk, V.; Zhu, X. Infinite rank module categories over finite 
dimensional $\mathfrak{sl}_2$-modules in Lie-algebraic context.
Preprint arXiv:2405.19894.

\bibitem[MZ25]{MZ2}
Mazorchuk, V.; Zhu, X. Combinatorics of infinite rank module 
categories over finite dimensional $\mathfrak{sl}_3$-modules 
in Lie-algebraic context.
J. Pure Appl. Algebra {\bf 229} (2025), no.~9, Paper No. 108054.

\bibitem[Mc80]{Mc}
McKay, J.
Graphs, singularities, and finite groups. 
In: The Santa Cruz Conference on Finite Groups, Proc. 
Sympos. Pure Math. {\bf 37}, Amer. Math. Soc., 
Providence, 1980, 183--186.

%

\bibitem[OS]{OS}
Ostrovskii, V.; Samoilenko, Yu.
On pairs of operators connected by a quadratic relation.
Funktsional. Anal. i Prilozhen. 47 (2013), no. 1, 82--87; 
translation in Funct. Anal. Appl. 47 (2013), no. 1, 67--71.

\bibitem[Re02]{Re}
Reid, M. La correspondance de McKay.
Ast{\'e}risque No. {\bf 276} (2002), 53--72.

%
%

\bibitem[Sm70]{Sm}
Smith, J. Some properties of the spectrum of a graph.
Gordon and Breach Science Publishers, New York-London-Paris, 
1970, pp. 403--406.

\bibitem[St17]{St}
Stevens, J.
Simple surface singularities.
Algebr. Geom. {\bf 4} (2017), no. 2, 160--176.


\end{thebibliography}
\end{document}